\documentclass[12pt,english,a4paper]{article}
\usepackage[centertags]{amsmath}
\usepackage{amssymb}
\usepackage{amsthm}
\usepackage{makeidx}

\newtheorem{theorem}{Theorem}[section]
\newtheorem{corollary}[theorem]{Corollary}
\newtheorem{lemma}[theorem]{Lemma}
\newtheorem{proposition}[theorem]{Proposition}
\newtheorem{definition}[theorem]{Definition}
\newtheorem{remark}[theorem]{Remark}
\newtheorem{example}[theorem]{Example}

\numberwithin{equation}{section} 

\begin{document}

\title{Around operators not increasing the degree of polynomials}

\author{T. Augusta Mesquita\footnote{Corresponding author (teresam@portugalmail.pt ; teresa.mesquita@fc.up.pt)} ,  P. Maroni }
\date{}
\maketitle

\begin{center}
{\scriptsize
Instituto Polit\'ecnico de Viana do Castelo, Av. do Atl\^antico, 4900-348 Viana do Castelo, Portugal \& \\
Centro de Matem\'{a}tica da Univ. do Porto, Rua do Campo Alegre, 687, 4169-007 Porto, Portugal\\
\medskip
\medskip
CNRS, UMR 7598, Laboratoire Jacques-Louis Lions, F-75005, Paris, France \& \\
UPMC Univ Paris 06, UMR 7598, Lab. Jacques-Louis Lions,
F-75005, Paris, France;\\
}
\end{center}

\begin{abstract}
We present a generic operator $J$ simply defined as a linear map not increasing the degree from the vectorial space of polynomial functions into itself and we address the problem of finding the polynomial sequences that coincide with the (normalized) $J$-image of themselves. The technique developed assembles different types of operators  and initiates with a transposition of the problem to the dual space. It is also provided examples where the results are applied to the case where $J$'s expansion is limited to three terms.
\end{abstract}

 \textbf{Keywords and phrases}: classical orthogonal polynomials \,,\, differential operators \,,\, Appell polynomial sequences \,,\, two-orthogonal polynomials.\\
 \textbf{2010 Mathematics Subject Classification}: 42C05 \,,\,  33C45 \,,\, 33D45

\section*{Introduction}

Classical orthogonal polynomials satisfy very well known 
differential equations of the second order with polynomial coefficients (e.g. \cite{Hahn2, Chihara, Everitt-kwon-Littlejohn-W, euler}) and their different families emerge from a thorough description of such differential equations. Besides this  well settled characterization of the classical orthogonal polynomials, other investigations deal with diverse differential relations fulfilled by sequences of orthogonal polynomials.
\newline For example, in \cite{Dattoli2001} Dattoli \textit{et al.} employ the monomiality principle in order to obtain a differential equation fulfilled by each polynomial $p_{n}(x)$ and ultimately attain an operational definition of a given orthogonal sequence from which interesting properties are deduced. In particular, the monomiality principle
combines the use of two operators written in terms of the derivative operator and its inverse so that these two operators play similar roles to those of differentiation and multiplication by $x$ on monomials, respectively.
\newline In other important contributions like \cite{kwon}, Kwon \textit{et al.} consider a generic  differential operator $L_{N}[y]=\sum_{i=1}^{N}a_{i}(x)y^{(i)}(x)$ and discuss the existence of orthogonal polynomial sequences $\{ P_{n} \}_{n\geq 0}$, with $\deg\left( P_{n}(x)  \right)= n$, such that $L_{N}[P_{n}](x)=\lambda_{n}P_{n}$, for each $n$. It is then worth notice that in view of the actual state of art,  with regard to differential operators $L$ such that $\deg\left(L\left(P_{n}\right)   \right)= n $, operating into an orthogonal sequence, we have already some acute results, as for instance, the nonexistence of orthogonal solutions among such differential equations of odd order \cite{kwon}. 
More recently, some incursions on the study of an Appell-type behavior of orthogonal polynomials \cite{QuadraticAppell, MesquitaPMH} allowed to gain new insights concerning differential relations fulfilled by orthogonal polynomial sequences.

In the present paper we aim to address a generic problem of determining the polynomial sequences fulfilling a differential relation based upon some of the utensils established by the authors in prior reference swhere the dual sequence has a major role.
\newline
Moreover, hereafter we become endowed with a generic set of results around a wider operator $J$ which can simply be described as a linear map not increasing the degree from the vectorial space of polynomial functions into itself. In the pages ahead, we shall be dealing with monic polynomial sequences $\{P_{n} \}_{n \geq 0}$ such that $ \textbf{P} = \Lambda\, J\left( \textbf{P}\right)$, with $\textbf{P}= \left(P_{0}, P_{1}, \ldots   \right)$, for a certain matrix $\Lambda$ gathering normalization coefficients. The method suits nicely  whether $J$ is an isomorphism  or when it imitates the usual derivative of order $k$, being $k$ any positive integer.  

After having reviewed the fundamental background in section 1,  the second and third sections present the characteristics of an operator $J$, firstly defined generally by an operator which do not increases the degree of polynomials and secondly focusing on the ones that imitate the behavior of a standard derivative of order $k$, for any non negative integer $k$.
In the last sections we apply the techniques previously outlined to an operator $J$ whose expansion is restricted to three initial terms considering that $\{P_{n} \}_{n \geq 0}$ is an orthogonal sequence. As an introduction to the analysis of multiple orthogonality we also deal with a two-orthogonal MPS transformed by a lowering operator, that is, an operator acting as a first order derivative.

\section{Basic definitions and notation}

Let $\mathcal{P}$ denote the vector space  of polynomials with coefficients in  $\mathbb{C}$ and let $\mathcal{P}'$ be its dual. We indicate by $\langle u,p \rangle$ the action of the form or linear functional  $u \in \mathcal{P}'$ on $p \in \mathcal{P}$.
\newline In particular, $(u)_{n}= \langle u,x^{n} \rangle,\: n \geq 0$, are called the moments of $u$. A form $u$ is equivalent to the numerical sequence  $\{(u)_{n}\}_{n \geq 0}$.

In the sequel, we will call polynomial sequence (PS) to any sequence 
${\{P_{n}\}}_{n \geq 0}$ such that $\deg P_{n}= n,\; \forall n \geq 0$.
We will also call monic polynomial sequence (MPS) a PS so that all polynomials have leading coefficient equal to one. Notice that if $\langle u,P_{n} \rangle=0,\; \forall n \geq 0$, then $u=0$. Given a MPS ${\{P_{n}\}}_{n \geq 0}$, there are complex sequences, ${\{\beta_{n}\}}_{n \geq 0}$ and
$\{\chi_{n,\nu}\}_{0 \leq \nu \leq n,\; n \geq 0},$ such that
\begin{align}
&P_{0}(x)=1, \;\; P_{1}(x)=x-\beta_{0},\label{divisao_ci}\\
&P_{n+2}(x)=(x-\beta_{n+1})P_{n+1}(x)-\sum_{\nu=0}^{n}\chi_{n,\nu}P_{\nu}(x).\label{divisao}
\end{align}
This relation is often called the structure relation of  ${\{P_{n}\}}_{n \geq 0}$, and ${\{\beta_{n}\}}_{n \geq 0}$ and
$\{\chi_{n,\nu}\}_{0 \leq \nu \leq n,\; n \geq 0}$ are called the structure coefficients.
Moreover, there exists a unique sequence  ${\{u_{n}\}}_{n \geq 0},\;\;
u_{n} \in \mathcal{P}'$, called the dual sequence of  ${\{P_{n}\}}_{n \geq 0}$, such that
$$\langle u_{n},P_{m} \rangle= \delta_{n,m},\;\; n,m \geq 0,$$
where $ \delta_{n,m}$ denotes the Kronecker symbol. Let us remark that, if $p$ is a polynomial and $ \langle u_{n},p \rangle=0,\; \forall n \geq 0$, then $p=0$.
Besides, it is well known that \cite{variations}
\begin{align}
&\label{betas}\beta_{n}= \langle u_{n},xP_{n}(x) \rangle,\:\: n \geq 0,\\
&\label{quisnus}\chi_{n,\nu}= \langle u_{\nu},xP_{n+1}(x) \rangle,\:\: 0\leq \nu \leq n,\:\: n \geq 0.
\end{align}
\begin{lemma}\cite{variations}\label{lema1}
For each $u \in \mathcal{P}^{\prime}$ and each $m \geq 1$, the two following propositions are equivalent.
\begin{description}
\item[a)]  $\langle u, P_{m-1} \rangle \neq 0,\:\: \langle u, P_{n} \rangle =0,\: n \geq m$.
\item[b)] $\exists \lambda_{\nu} \in \mathbb{C},\:\: 0 \leq \nu \leq m-1,$ such that
$u=\displaystyle\sum_{\nu=0}^{m-1}\lambda_{\nu} u_{\nu}$, with $\lambda_{m-1}\neq 0$. In particular, $\lambda_{\nu} = \langle u, P_{\nu} \rangle$. \end{description}
\end{lemma}
Given $\varpi \in \mathcal{P}$ and $u \in \mathcal{P}'$, the form $\varpi u$, called the
left-multiplication of $u$ by the polynomial $\varpi$, is defined by 
\begin{equation}\label{poly-times-u}
\langle \varpi u,p \rangle= \langle u,\varpi p \rangle,\:\:\: \forall p \in \mathcal{P},
\end{equation}
and the transpose of the derivative operator on $\mathcal{P}$ defined by $p \rightarrow (Dp)(x)=p^{\prime}(x),$ is the following (cf. \cite{theoriealgebrique}):
\begin{equation}\label{funcionalDu}
u \rightarrow Du:\;\;\langle Du,p \rangle=-\langle u,p^{\prime} \rangle ,\:\:\: \forall p \in \mathcal{P},\end{equation}
so that we can retain the usual rule of the derivative of a product when applied to the left-multiplication of a form by a polynomial. Indeed, it is easily established that 
\begin{align}
\label{derivada do produto} &D(pu)=p^{\prime}u+pD(u).
\end{align}
\begin{definition} \cite{theoriealgebrique,euler} \label{orthogonality definition}
A PS ${\{P_{n}\}}_{n \geq 0}$ is regularly orthogonal with respect to the form $u$ if and only if it fulfils 
\begin{align}
\label{ortogonal} &\langle u,P_{n}P_{m} \rangle=0,\:\: n \neq m,\:\:\:\:\: n, m \geq 0,\\
\label{ortogonal regular} & \langle u,P_{n}^{2} \rangle \neq 0, \: n \geq 0.
\end{align}
Then, the form $u$ is said to be regular (or quasi-definite) and ${\{P_{n}\}}_{n \geq 0}$ is an orthogonal polynomial sequence (OPS). The conditions (\ref{ortogonal}) are called the orthogonality conditions and the conditions (\ref{ortogonal regular}) are called the regularity conditions.
\end{definition}
We can normalize ${\{P_{n}\}}_{n \geq 0}$ so that it becomes monic, so that it is unique and we briefly note it as a MOPS. Considering the corresponding dual sequence ${\{u_{n}\}}_{n \geq 0}$, it holds $u=\lambda u_{0},$ with $\lambda=(u)_{0} \neq 0$.
\begin{lemma}\cite{euler}\label{phi-u=0}
Let $u$ be a regular form and $\phi$ a polynomial such that $\phi u=0$. Then $\phi=0$.
\end{lemma}
\begin{theorem}\cite{variations}\label{regular}
Let ${\{P_{n}\}}_{n \geq 0}$ be a MPS and ${\{u_{n}\}}_{n \geq 0}$ its dual sequence. 
The following statements are equivalent:
\begin{description}
  \item[a)]  The sequence ${\{P_{n}\}}_{n \geq 0}$ is orthogonal (with respect to $u_{0}$);
  \item[b)] $\chi_{n,k}=0$,  $\;0\leq k \leq n-1,\;\;\;\; n\geq 1; \;\;\;\chi_{n,n}\neq 0,\; \; n \geq
  0$;
  \item[c)]$xu_{n}=u_{n-1}+\beta_{n}u_{n}+\chi_{n,n}u_{n+1}$, $\chi_{n,n}\neq
  0,\;\;\; n \geq 0$, $u_{-1}=0$;
  \item[d)] For each $n \geq 0$, there is a polynomial  $\phi_{n}$ with $\deg(\phi_{n})=n$ such that $u_{n}=\phi_{n}u_{0}$;
  \item[e)] $u_{n}=\Big(<u_{0},P_{n}^{2}>\Big)^{-1}P_{n}u_{0}$, $n \geq 0$;
\end{description}
where $\beta_{n}$ and $\chi_{n,k}$ are defined by (\ref{betas}-\ref{quisnus}).
\end{theorem}
Let ${\{P_{n}\}}_{n \geq 0}$ be a MOPS. From statement $b)$ of Theorem \ref{regular}, the structure relation (\ref{divisao})
becomes the following second order recurrence relation: 
\begin{align}\label{recurrencia de ordem dois}
&P_{0}(x)=1,\;\;P_{1}(x)=x-\beta_{0},\\
&P_{n+2}(x)=(x-\beta_{n+1})P_{n+1}(x)-\gamma_{n+1}P_{n}(x),\;\;\; n \geq 0,
\end{align}
where $\gamma_{n+1}=\chi_{n,n} \neq 0,\;\; n \geq 0$, and also by item $e)$, we have:
\begin{equation}\label{betas e gammas}
\beta_{n}=\frac{\langle u_{0}, xP_{n}^{2}(x) \rangle}{ \langle u_{0}, P_{n}^{2}(x) \rangle},\;\;\: \:\:\:\gamma_{n+1}=\frac{ \langle u_{0}, P_{n+1}^{2}(x) \rangle}{ \langle u_{0}, P_{n}^{2}(x) \rangle},
\end{equation}
being the regularity conditions (\ref{ortogonal regular}) fulfilled if and only if $\gamma_{n+1} \neq 0,\:\: n \geq 0$. Notice also that $\gamma_{1}\ldots\gamma_{n}= \prod_{i=1}^{n} \gamma_{i} =\langle u_{0}, P_{n}^{2}(x) \rangle,\; n \geq 1$.
\newline The use of suitable affine transformations requires the use of the following operators on $\mathcal{P}$ \cite{theoriealgebrique}:
\begin{align*}
&p \rightarrow \tau_{B}\,p(x)=p(x-B),\:\: B \in \mathbb{C},\\
&p \rightarrow h_{A}\,p(x)=p(Ax),\:\: A \in \mathbb{C} \backslash \{0\}.
\end{align*}
Transposing, we obtain the corresponding operators on $\mathcal{P}^{\prime}$.
\begin{align*}
&u \rightarrow \tau_{b}u:\;\;<\tau_{b}u,p>=<u,\tau_{-b}p>=<u,p(x+b)>,\:\:\: \forall p \in \mathcal{P},\\
&u \rightarrow h_{a}u:\;\;<h_{a}u,p>=<u,h_{a}p>=<u,p(ax)>,\:\:\: \forall p \in \mathcal{P}.
\end{align*}
Hence, given $A \in \mathbb{C}\backslash \{0\}$ and $B \in \mathbb{C}$, and a MPS ${\{P_{n}\}}_{n \geq 0}$ we may define the outcome of an affine transformation denoted by ${\{\tilde{P}_{n}\}}_{n \geq 0}$  as follows:
\begin{equation}\label{shifted sequence}
\tilde{P}_{n}(x)=A^{-n}P_{n}(Ax+B), \:\: n \geq 0.
\end{equation}
with dual sequence \cite{variations}:
$$\tilde{u}_{n}=A^{n}(h_{A^{-1}}\circ \tau_{-B})u_{n}.$$
In particular, if ${\{P_{n}\}}_{n \geq 0}$ is a MOPS, then the MPS defined by \eqref{shifted sequence} is orthogonal 
 and its recurrence coefficients are 
 \begin{equation}\label{shifted-coef}
 \widetilde{\beta}_{n}=\frac{\beta_{n}-B}{A},\;\; \widetilde{\gamma}_{n+1}=\frac{\gamma_{n+1}}{A^{2}},\;\;\;n \geq 0.
 \end{equation}
Finally, we recall that a MPS ${\{P_{n}\}}_{n \geq 0}$ is called classical, if and only if it satisfies the Hahn\'{}s property \cite{Hahn}, that is to say, the MPS ${\{P^{[1]}_{n}\}}_{n \geq 0}$, defined by $P^{[1]}_{n}(x):=(n+1)^{-1}DP_{n+1}(x)$, is also orthogonal. The classical polynomials are divided into four classes: Hermite, Laguerre, Bessel and Jacobi \cite{Chihara}, and characterized by the functional equation
\begin{equation}\label{classical-equation}
D(\phi u)+\psi u =0,
\end{equation}
where $\psi$ and $\phi$ are two polynomials such that: $\deg \psi =1$, $\deg \phi \leq 2$, $\phi$ is normalized and $\psi^{\prime}-\frac{1}{2}\phi^{\prime\prime}n \neq 0,\: n \geq 1$ \cite{euler}. In fact, since $\phi$ cannot be identically zero, otherwise $u_{0}$ would not be regular, we consider it monic and the same for the form $u$, that is, $(u)_{0}=1$. For example, the polynomials $\phi(x)=x$ and $\psi(x)=x-\alpha-1$, with parameter $\alpha \notin \mathbb{Z}^{-}$, correspond to the Laguerre polynomials.
\newline Furthermore, when we apply an affine transformation to a classical MOPS, orthogonal with respect to $u_{0}$, as written in \eqref{shifted sequence}, we obtain also a classical MOPS orthogonal with respect to the form $\tilde{u}_{0}$, defined by $\tilde{u}_{0}=(h_{A^{-1}}\circ \tau_{-B})u_{0}$ and belonging to the same class \cite{theoriealgebrique,euler}. In addition, $\tilde{u}_{0}$ fulfils $D(\tilde{\phi} u)+\tilde{\psi} u =0$ where \cite{theoriealgebrique}
\begin{align}\label{phi-psi-affine} 
\tilde{\phi}(x)=A^{-t}\phi\left(  Ax+B \right),\quad \tilde{\psi}(x)=A^{1-t}\psi\left(  Ax+B \right),\quad t=\deg\left(\phi \right).   
\end{align}

\section{Operators on $\mathcal{P}$ and technical identities}\label{operator}

Given a sequence of polynomials $\{a_{\nu}(x)\}_{ \nu \geq 0}$, let us consider the following linear mapping $J: \mathcal{P} \rightarrow \mathcal{P}$ (cf. \cite{Korean}, \cite{Pincherle}).
\begin{equation} \label{operatorJ}
J= \sum_{\nu \geq 0} \frac{a_{\nu}(x)}{\nu!} D^{\nu},\quad \deg a_{\nu} \leq \nu,\quad \nu \geq 0.
\end{equation}
Expanding $a_{\nu}(x)$ as follows:
$$a_{\nu}(x)=\sum_{i=0}^{\nu} a_{i}^{[\nu]}x^{i}, $$
and in view of $ D^{\nu}\left( \xi ^n \right) (x)= \frac{n!}{(n-\nu)!} x^{n-\nu} $, we get the next useful identities about $J$:
\begin{eqnarray}
\label{Jx^n-short}&& J\left( \xi^n \right) (x) = \sum_{\nu= 0}^{n} a_{\nu} (x) 
\binom{n}{\nu} x^{n-\nu}\\
\label{Jx^n} && J\left( \xi^n \right) (x) = \sum_{\tau= 0}^{n} \left( \sum_{\nu=0}^{\tau}\binom{n}{n-\nu} a_{\tau-\nu}^{[n-\nu]} \right) x^{\tau},\quad n \geq 0.
\end{eqnarray}
In particular, a linear mapping $J$ is an isomorphism if and only if 
\begin{equation}\label{iso-conditions}
\deg \left( J\left( \xi^n  \right)(x) \right)= n\;, \;\; n \geq 0, \; \; \textrm{and}\;\; J\left( 1  \right)(x) \neq 0.
\end{equation}

\begin{lemma}\label{lemmaPascal}
For any linear mapping $J$, not increasing the degree, there exists a unique sequence of polynomials $\{a_{n}\}_{n \geq 0}$, with $\deg a_{n} \leq n$, so that $J$ is read as in \eqref{operatorJ}. Further, the linear mapping $J$ is an isomorphism of $\mathcal{P}$ if and only if
\begin{equation}\label{lambdan}
 \sum_{\mu=0}^{n}\binom{n}{\mu} a_{\mu}^{[\mu]} \neq 0 , \quad n \geq 0. 
 \end{equation}
\end{lemma}
\begin{proof}
The linear mapping is well determined by the data of images $J\left( \xi^n  \right)(x) = D_{n}(x),\; n \geq 0$, where $D_{n}(x)=\displaystyle\sum_{\tau=0}^{n} d_{\tau}^{[n]} x^{\tau}$.
Comparing with \eqref{Jx^n}, we obtain
\begin{align*}
& d_{\tau}^{[n]} = \sum_{\nu=0}^{\tau}  \binom{n}{n-\nu} a_{\tau-\nu}^{[n-\nu]}.
\end{align*}
Therefore, for given coefficients $d_{\tau}^{[n]} ,\;  0\leq \tau \leq n ,\; n \geq 0$, the recurrence relation
\begin{equation}\label{recurrence of J}
a_{\tau}^{[n]} = d_{\tau}^{[n]}-\sum_{\nu=1}^{\tau}  \binom{n}{n-\nu} a_{\tau-\nu}^{[n-\nu]},\quad 1 \leq \tau \leq n,
\end{equation}
with initial conditions $a_{0}^{[n]} = d_{0}^{[n]} , \; n \geq 0$, establishes the computation of a unique set of coefficients $a_{\tau}^{[n]},\; 0\leq \tau \leq n ,$ for any $n \geq 0$.
\end{proof}

For any given operator $J$, not increasing the degree, we may calculate through \eqref{recurrence of J} its development in a $J$ format. This task can be accomplished by cumbersome calculations, but also by the use of a   
computer algebra software like \textit{Mathematica} \cite{wolfram}, where \eqref{recurrence of J} has a simple redaction and the first $nmax$ polynomial coefficients $a_{n}(x),\; n=0,\ldots,nmax,$ are easily calculated, for a chosen positive integer $nmax$. The next examples were derived from both processes.

\begin{example}
 When $J=D$, we have $a_{n}(x)=\delta_{n,1},\; n \geq 0$.
\end{example}
\begin{example}
If $J=DxD$, then $J\left(\xi^n \right)(x)=n^2 x^{n-1},\; n \geq 0$. Consequently, we have $a_{0}(x)=0$, $a_{1}(x)=1$, $a_{2}(x)=2x$, $a_{n}(x)=0,\; n \geq 3$. Thus, $J=D+xD^2$.
\end{example}
\begin{example}\label{affineJ}
When $J=s\, \left(h_{A} \circ \tau_{-B} \right)$, $ s \neq 0$, we have $J\left(\xi^n \right)(x)=s \left( Ax+B\right)^{n},$ $\; n \geq 0$, therefore from \eqref{recurrence of J}, we easily obtain $a_{n}(x)=s \left( \left(A-1\right)x+B \right)^{n},\; n \geq 0$. Then $J = s \displaystyle\sum_{n \geq 0}  \frac{ \left( \left(A-1\right)x+B \right)^{n}}{n!}D^{n}$.
\end{example}
\begin{example}\label{DvarpiJ}
If  $J=D_{\varpi}$, where $\left( D_{\varpi}f \right)(x)=\varpi^{-1}\left( f(x + \varpi) -f(x)\right),$ $ f \in \mathcal{P}$, $\varpi \in \mathbb{C}-\{ 0\}$ \cite{Dw-classical}, then 
$$J\left( 1 \right) (x)=0\, , \quad J\left( \xi^n \right) (x)= \sum_{\tau=0}^{n-1} \binom{n}{\tau} \varpi^{n-1-\tau} x^{\tau}\,,\quad n \geq 1,$$
implying $a_{0}(x)=0$ and $a_{n}(x)=\varpi^{n-1},\; n \geq 1$, so that
$$J= \sum_{n \geq 0} \frac{\varpi^{n}}{(n+1)!}D^{n+1}.$$
\end{example}
\begin{example}\label{HqJ}
Let us consider  $J=H_{q}$, where $\left( H_{q}f \right)(x)=(q-1)^{-1}x^{-1}\left( f(qx) -f(x)\right),$ $ f \in \mathcal{P}$, $q \in \mathbb{C}-\bigcup_{n \geq 0} U_{n}$, with $U_{0}=\{ 0 \}$,  $U_{n}=\{ z \in \mathbb{C} \,|\, z^{n} =1 \},\; n \geq 1$ \cite{Hq-classical}. Then $J\left( \xi^n \right) (x)= \dfrac{q^n-1}{q-1} x^{n-1}\,,\; n \geq 0$, and some computations yield $a_{0}(x)=0$ and $a_{n}(x)=(q-1)^{n-1} x^{n-1}\,,\; n \geq 1$. Consequently
$$J= \sum_{n \geq 0} \frac{ (q-1)^{n}}{(n+1)!}x^{n}D^{n+1}.$$
\end{example}
\begin{example}\label{IqomegaJ}
Considering  $J=I_{(q,\omega)}$, where $\left( I_{(q,\omega)} f \right)(x)=f(x) + \omega f(qx),$ $ f \in \mathcal{P}$, $\omega \in \mathbb{C}-\{ 0\}$,  $q \in \mathbb{C}_{\omega}$ with  $\mathbb{C}_{\omega}=\{ z \in \mathbb{C} :  z \neq 0, z^{n+1} \neq 1,\; 1+\omega z^{n} \neq 0,\; n \in  \mathbb{N} \}$ \cite{Iqw-classical}, we get  $a_{0}(x)=1 +\omega$, $a_{n}(x)=\omega \left( q-1  \right)^n x^n\, , \; n \geq 1$, and 
 $$J= I+ \omega \sum_{n \geq 0} \frac{ (q-1)^{n}}{n!}x^{n}D^{n}.$$
\end{example}

\medskip
Coming back to a generic operator $J$, by duality and in view of \eqref{funcionalDu}, we have
\begin{align*}
\langle ^{t}J(u), f  \rangle &= \langle u, J(f) \rangle \;, \quad u \in \mathcal{P}^{\prime},\quad f \in \mathcal{P},\\
& = \sum_{n \geq 0} \langle u, \frac{a_{n}(x)}{n!}f^{(n)}(x) \rangle = \sum_{n \geq 0} \frac{(-1)^n}{n!} \langle D^{n}\left(a_{n}u\right), f \rangle ;
\end{align*} 
thence
\begin{equation}\label{transpose-J}
^{t}J(u) = \sum_{n \geq 0} \frac{(-1)^n}{n!}  D^{n}\left(a_{n}u\right), \quad u \in \mathcal{P}^{\prime}.
\end{equation}
We will denote the transposed operator $^{t}J(u)$ more simply by $J(u)$, since in each context it is distinguishable from $J$ on $\mathcal{P}$.

Still regarding the operator $J$ and its underlying sequence of polynomials $\{a_{n}\}_{n\geq 0}$,  let us now consider the following sum.
\begin{equation}
J(x;z) := \sum_{n \geq 0} \frac{a_{n}(x)}{n!}  z^{n} .
\end{equation}
Differentiating with respect to $z$, we obtain
\begin{equation}
J^{\prime}(x;z) := \frac{\partial}{\partial z} J(x; z) = \sum_{n \geq 0 } \frac{a_{n+1}(x)}{n!}z^n ;
\end{equation}
and more generally, we may define
\begin{equation}
J^{(m)}(x;z) = \sum_{n \geq 0 } \frac{a_{n+m}(x)}{n!}z^n , \quad m \geq 0.
\end{equation}
The transpose operator of the operator on $\mathcal{P}$ defined by $J^{(m)}= \displaystyle \sum_{n \geq 0 } \frac{a_{n+m}(x)}{n!}D^{n}$ is the following.
\begin{equation}\label{J^(m)}
J^{(m)}(u) = \sum_{n \geq 0 } \frac{(-1)^n}{n!}D^n\left( a_{n+m} u \right), \quad m \geq 0.
\end{equation}

\begin{example}
With respect to the examples above, we get:
\begin{align*}
& J(x;z) = s\, \exp\left( \left(  \left(A-1 \right)x+B \right)z\right)\, , \quad \textrm{for the operator of example} \;\; \ref{affineJ} , \\
& J(x;z) = \varpi^{-1}\left( \exp\left(  \varpi z \right) -1  \right)\, , \quad \textrm{for the operator of example} \;\; \ref{DvarpiJ} , \\
& J(x;z) = \left( (q-1) x \right)^{-1}\left( \exp\left(  (q-1)x z \right) -1  \right)\, , \quad \textrm{for the operator of example} \;\; \ref{HqJ}, \\
&  J(x;z) = 1 + \omega \exp\left(  (q-1)x z \right)\, , \quad \textrm{for the operator of example} \;\; \ref{IqomegaJ}.
\end{align*}
\end{example}

\begin{lemma} For any $f,g \in \mathcal{P}$, $u \in \mathcal{P}^{\prime}$, we have:
\begin{align}
\label{J(fg)}&J\left( fg \right)(x)= \sum_{n \geq 0 } J^{(n)}\left(f\right)(x)  \frac{g^{(n)}(x)}{n!} = \sum_{n \geq 0 } J^{(n)}\left(g\right)(x)  \frac{f^{(n)}(x)}{n!},\\
\label{J(fu)}&J\left( fu  \right) = \sum_{n \geq 0 } \frac{(-1)^n}{n!}f^{(n)} J^{(n)}(u).
\end{align}
If we take another operator $K=\displaystyle  \sum_{n \geq 0 } \frac{b_{n}(x)}{n!}D^{n}$, we may then consider
\begin{align}
&K \circ J = \sum_{n \geq 0} \frac{c_{n}(x)}{n!}D^{n}, \quad \textrm{where}
\end{align}
\begin{equation}\label{composed-coeff}
c_{n}(x)= \displaystyle \sum_{\nu =0}^{n} \binom{n}{\nu} \sum_{\mu=0}^{\nu} \frac{b_{n-\nu+\mu}(x)}{\mu!} a_{\nu}^{\left(\mu\right)}(x)= \sum_{\mu=0}^{n} b_{\mu}(x) \sum_{\nu=0}^{\mu} \binom{n}{\nu} \frac{a_{n-\nu}^{(\mu-\nu)}(x)}{(\mu-\nu)!}, \; n \geq 0.
\end{equation}
\end{lemma}
\begin{proof} The following calculations justify \eqref{J(fg)}.
\begin{align*}
J\left( fg \right)(x) &= \sum_{\nu \geq 0 }  \frac{a_{\nu}(x)}{\nu!} \left( fg  \right)^{(\nu)}(x)  =  \sum_{\nu \geq 0 }  \frac{a_{\nu}(x)}{\nu!} \sum_{\mu=0}^{\nu}  \binom{\nu}{\mu} f^{\left( \mu \right)}(x) g^{\left(\nu-\mu  \right)}(x) \\
&= \sum_{\mu \geq 0 } f^{\left( \mu  \right)}(x) \sum_{\nu \geq \mu} \frac{a_{\nu}(x)}{\nu!} \binom{\nu}{\mu} g^{\left( \nu-\mu  \right)}(x)\\
&= \sum_{\mu \geq 0 } f^{\left( \mu  \right)}(x) \sum_{\tau \geq 0} \frac{a_{\mu+\tau}(x)}{\left( \mu+\tau \right)!} \binom{\mu + \tau}{\mu} g^{\left( \tau  \right)}(x)\\
  & = \sum_{\mu \geq 0 }\left(  \sum_{\tau \geq 0 } \frac{a_{\mu+\tau}(x)}{\tau!}  g^{\left(\tau  \right)}(x)  \right) \frac{f^{\left( \mu  \right)}(x)}{\mu!} = \sum_{\mu \geq 0 } J^{\left( \mu  \right)}\left( g \right)\frac{f^{\left( \mu  \right)}(x)}{\mu!} .
\end{align*}
Next, we have
\begin{align*}
J\left( f u  \right) & = \sum_{\nu \geq 0 } \frac{(-1)^{\nu}}{\nu!} D^{\nu}\left(a_{\nu} f u \right) \\
 & = \sum_{\nu \geq 0} \frac{(-1)^{\nu}}{\nu!} \sum_{\mu=0}^{\nu} \binom{\nu}{\mu} f^{\left(\nu-\mu  \right)}\left(a_{\nu} u \right)^{\left( \mu   \right)} \\
& = \sum_{\mu \geq 0} \sum_{\nu \geq \mu} \frac{(-1)^{\nu}}{\nu!} \binom{\nu}{\mu} f^{\left(\nu-\mu  \right)}\left(a_{\nu} u \right)^{\left( \mu   \right)} \\
 & = \sum_{\mu \geq 0} \sum_{\tau \geq 0}  \frac{(-1)^{\mu+\tau}}{\left( \mu+\tau \right)!} \binom{\mu+\tau}{\mu} f^{\left( \tau  \right)} \left( a_{\mu+\tau} u   \right)^{\left( \mu  \right)} \\
 & = \sum_{\mu \geq 0} \frac{(-1)^{\mu}}{\mu!} \sum_{\tau \geq 0} \frac{(-1)^{\tau}}{\tau!} f^{\left( \tau  \right)}  \left( a_{\mu+\tau} u   \right)^{\left( \mu  \right)} \\
 & = \sum_{\tau \geq 0} \frac{(-1)^{\tau}}{\tau!} f^{\left( \tau  \right)} \sum_{\mu \geq 0}  \frac{(-1)^{\mu}}{\mu!}     \left( a_{\mu+\tau} u   \right)^{\left( \mu  \right)} = \sum_{\tau \geq 0} \frac{(-1)^{\tau}}{\tau!} f^{\left( \tau  \right)} J^{\left(\tau   \right)}(u),
\end{align*}
whence \eqref{J(fu)}.
Finally, being $\deg f = m$, we have
\begin{align*}
\left( K\left( J f  \right) \right) (x) & = \sum_{\mu = 0 }^{m} \frac{b_{\mu}(x)}{\mu!} \left(\sum_{\nu=0}^{m}  \frac{a_{\nu}(x)}{\nu!} f^{(\nu)}\right)^{(\mu)}\\
 & = \sum_{\mu = 0 }^{m} \frac{b_{\mu}(x)}{\mu!} \sum_{\nu=0}^{m}  \frac{1}{\nu!} \left( a_{\nu}f^{(\nu)}\right)^{(\mu)}\\
 & = \sum_{\mu = 0 }^{m} \frac{b_{\mu}(x)}{\mu!} \sum_{\nu=0}^{m} \frac{1}{\nu!} \sum_{\tau =0}^{\mu} \binom{\mu}{\tau} a_{\nu}^{(\mu-\tau)}(x)f^{(\nu+\tau)}(x)\\
 & = \sum_{\nu=0}^{m} \frac{1}{\nu!} \sum_{\mu = 0 }^{m} \frac{b_{\mu}(x)}{\mu!}\sum_{\tau =0}^{\mu} \binom{\mu}{\tau} a_{\nu}^{(\mu-\tau)}(x)f^{(\nu+\tau)}(x)
 \end{align*}
 \begin{align*}
\left( K\left( J f  \right) \right) (x)  & = \sum_{\nu=0}^{m} \frac{1}{\nu!} \sum_{\tau=0}^{m} f^{(\nu+\tau)}(x)\sum_{\mu=\tau}^{m} \binom{\mu}{\tau}\frac{b_{\mu}(x)}{\mu!}a_{\nu}^{(\mu-\tau)}(x)\\
 & = \sum_{n=0}^{m}\left( \sum_{\nu+\tau=n} \frac{1}{\nu!} \sum_{\mu=\tau}^{\nu+\tau} \binom{\mu}{\tau}\frac{b_{\mu}(x)}{\mu!}a_{\nu}^{(\mu-\tau)}(x) \right) f^{(n)}(x)\\
 &= \sum_{n= 0 }^{m} \frac{c_{n}(x)}{n!}f^{(n)}(x), \quad \textrm{where}
 \end{align*}
\begin{align*}
c_{n}(x)& = n! \sum_{\nu+\tau = n}\frac{1}{\nu!} \sum_{\mu=\tau}^{n} \binom{\mu}{\tau} \frac{b_{\mu}(x)}{\mu!} a_{\nu}^{(\mu-\tau)}(x)=\sum_{\nu=0}^{n} \binom{n}{\nu} \sum_{\mu=0}^{\nu} \frac{b_{n-\nu+\mu}(x)}{\mu!}a_{\nu}^{(\mu)}(x)\\
& = \sum_{\nu=0}^{n} \binom{n}{\nu} \sum_{\mu=0}^{\nu} b_{n-\mu}(x) \frac{a_{\nu}^{(\nu-\mu)}(x)}{\left(\nu-\mu \right)!} = \sum_{\mu=0}^{n} b_{n-\mu}(x) \sum_{\nu=\mu}^{n} \binom{n}{\nu}\frac{a_{\nu}^{(\nu-\mu)}(x)}{\left(\nu-\mu \right)!}\\
&=\sum_{\mu=0}^{n} b_{\mu}(x) \sum_{\nu=0}^{\mu}\binom{n}{\nu}\frac{a_{n-\nu}^{(\mu-\nu)}(x)}{\left(\mu-\nu \right)!}\; , \quad n \geq 0. \qedhere
 \end{align*}\end{proof}

\begin{corollary}
When $J$ is an isomorphism, the inverse $J^{-1}$ is given by $J^{-1} = \sum_{n\geq 0} \dfrac{\tilde{a}_{n}(x)}{n!}D^{n},$ where
\begin{align*}
\tilde{a}_{0}(x)=\left( \lambda_{0}^{[0]}\right)^{-1},\;\; \lambda_{n+1}^{[0]}\tilde{a}_{n+1}(x) = - \sum_{\mu=0}^{n} \tilde{a}_{\mu}(x) \sum_{\nu=0}^{\mu} \binom{n+1}{\nu} \frac{a_{n+1-\nu}^{(\mu-\nu)}(x)}{(\mu-\nu)!}\;, \; n \geq 0, \quad \end{align*}
with  $\lambda_{n}^{[0]} := \sum_{\mu=0}^{n}\binom{n}{\mu} a_{\mu}^{[\mu]}$ that are assumed non-zero in \eqref{lambdan}.\end{corollary}
\begin{proof}
From \eqref{composed-coeff}, we have $1 = \tilde{a}_{0}(x)a_{0}(x)$ and for $n \geq 1$,
\newline $0 = \sum_{\mu=0}^{n}\tilde{a}_{\mu}(x)\sum_{\nu=0}^{\mu} \binom{n}{\nu} \frac{a_{n-\nu}^{(\mu-\nu)}(x)}{(\mu-\nu)!}$, or with $n \rightarrow n +1$
$$\tilde{a}_{n+1}(x)\sum_{\nu=0}^{n+1} \binom{n +1}{n+1-\nu} \frac{a_{\nu}^{(\nu)}(x)}{\nu!} =- \sum_{\mu=0}^{n}\tilde{a}_{\mu}(x)\sum_{\nu=0}^{\mu} \binom{n +1}{\nu} \frac{a_{n+1-\nu}^{(\mu-\nu)}(x)}{(\mu-\nu)!}\,,\; n \geq 0.$$
But $a_{\nu}^{(\nu)}(x)=\nu!a_{\nu}^{\nu}$, whence the result. 
\end{proof}

\section{Looking at $J$ as a lowering operator of order $k$}

The use of the generic operator $J$ permits to deal with problems of finding polynomial sequences either defined as (element-wise) eigenfunctions of a given operator, or fulfilling an Appel behavior towards such operator, that is, when $J$ implies the descending of the degree in precisely one, as the single derivative acts. 
The technical approach has been already tackled in previous contributions where the authors efforts were mainly focused on an Appel behavior (e.g.\cite{MesquitaPMH}). We shall now proceed with the introduction of a wider definition which assembles all the referred situations.

\medskip

Let us suppose that $J$ is an operator expressed as in \eqref{operatorJ}, and acting as the derivative of order $k$, for some non-negative integer $k$, that is, it fulfils the following conditions.
\begin{align}
\label{deg-k1} & J\left( \xi^{k}\right)(x) = a_{0}^{[k]} \neq 0 \;\; \textrm{and} \; \;\deg\left(J\left( \xi^{n+k}\right)(x)  \right) = n,\quad n \geq 0;  \\
\label{deg-k2}& J\left( \xi^{i}\right)(x) =0,\quad 0 \leq i \leq k-1,\; \textrm{if} \; k\geq 1.
\end{align}

\begin{lemma}\label{J-descending}
An operator $J$ fulfils \eqref{deg-k1}-\eqref{deg-k2} if and only if the next set of conditions hold.
\begin{description}
\item[a)] $a_{0}(x)=\cdots=a_{k-1}(x)=0$, if $k \geq 1$;
\item[b)] $\deg\left( a_{\nu}(x) \right) \leq \nu-k$, $\nu \geq k$;
\item[c)] \begin{equation}\label{neqcondition}
\lambda_{n+k}^{[k]} := \sum_{\nu = 0}^{n} \binom{n+k}{n+k-\nu} a_{n-\nu}^{[n+k-\nu]} \neq 0, \; n \geq 0.
\end{equation}
\end{description}
\end{lemma}
\begin{proof}
Firstly, let us prove that a) and \eqref{deg-k2} are equivalent. Recalling identity \eqref{Jx^n-short}, 
it is clear that $J\left( \xi^0 \right) (x) =0$ implies $a_{0}(x)=0$, and since
$$J\left( \xi^i \right) (x) = 0 \Leftrightarrow a_{i}(x) = - \sum_{\nu= 0}^{i-1} a_{\nu} (x) 
\binom{i}{\nu} x^{i-\nu},$$
we can prove, by induction on $i$, that \eqref{deg-k2} $\Rightarrow$ a), being  the opposite sense trivial. Let us remark again that by \eqref{Jx^n-short}, we have: 
$$J\left( \xi^{n+k}\right)(x)= \sum_{\nu= 0}^{n+k} a_{\nu} (x) 
\binom{n+k}{\nu} x^{n+k-\nu} = \sum_{\tau= 0}^{n+k}  \left( \sum_{\nu= 0}^{\tau}  \binom{n+k}{n+k-\nu} a_{\tau-\nu}^{[n+k-\nu]}    \right) x^{\tau}$$ where the coefficient of $x^n$ is given by the above defined constant $\lambda_{n+k}^{[k]}$.
\newline If we suppose b) then $\deg\left(J\left( \xi^{n+k}\right)(x)  \right) \leq n$, and in addition c) implies $\deg\left(J\left( \xi^{n+k}\right)(x)  \right) = n$.
Conversely, if we suppose \eqref{deg-k1}, then $\lambda_{n+k}^{[k]} \neq 0$ and 
$$\sum_{\nu= 0}^{n+s}  
\binom{n+k}{n+k-\nu} a_{n+s-\nu}^{[n+k-\nu]}=0,$$
for $s=1,\ldots, k$ and $n \geq 0$.
Rewriting this last identity, we get:
\begin{equation}\label{induction-eq}
a_{n+s}^{[n+k]} =-\sum_{\nu= 1}^{n+s}  
\binom{n+k}{n+k-\nu} a_{n+s-\nu}^{[n+k-\nu]}.
\end{equation}
Considering \eqref{induction-eq} with $n=0$ and applying item a), we obtain $a_{s}^{[k]}=0,\; s=1,\ldots,k,$ which implies $\deg \left( a_{k}(x) \right) \leq 0$, or item b) for $\nu=k$.
The induction hypotheses $\deg \left( a_{i}(x) \right) \leq i-k$, with $i=k,\ldots, k+n-1$, for  a certain $n \geq 1$, correspond to assume that $a_{j}^{[i]}=0,\;\; 0 \leq i-j \leq k-1, \quad i= k,\ldots, k+n-1$.
Let us now note that $0\leq n+k-\nu-(n+s-\nu) \leq k-1$ and insert the hypotheses on \eqref{induction-eq} in order to conclude $a_{n+s}^{[n+k]} =0,\; s=1,\ldots, k$, that is, $\deg(a_{n+k}) \leq n$, which completes the proof.
\end{proof}

\begin{remark}
Note that in \eqref{neqcondition} we find $\lambda_{k}^{[k]}=a_{0}^{[k]} $.
\newline If $k=0$, then it is assumed that $\lambda_{n}^{[0]} \neq 0,\; n \geq 0$, matching \eqref{lambdan}, so that $J$ is an isomorphism. 
\newline If $k=1$, then $J$ imitates the usual derivative and is commonly called a lowering operator \cite{QuadraticAppell, Srivastava-Ben cheik, MesquitaPMH}.
\end{remark}

\begin{definition}
An operator $J$ fulfilling \eqref{deg-k1}-\eqref{deg-k2}, or equivalently items a)-c) of Lemma \ref{J-descending}, for some positive integer $k$, is herein called lowering operator of order $k$.
\end{definition}

Given a MPS $\{ P_{n} \}_{n \geq 0}$ and a non-negative integer $k$, let us define its (normalized) $J$-image sequence of polynomials as follows, and notate its dual sequence by $\{ \tilde{u}_{n}\}_{n \geq 0}$. Of course, given that $J$ satisfies \eqref{deg-k1}-\eqref{deg-k2}, the obtained polynomial sequence $\{ \tilde{P}_{n} \}_{n \geq 0}$ is also a MPS.  
\begin{equation}\label{JPn}
\tilde{P}_{n}(x)= \left( \lambda_{n+k}^{[k]} \right)^{-1}J\left( P_{n+k} (x) \right), \quad n \geq 0.
\end{equation}
We keep using $J$ to indicate also the transpose operator  $\;^{t}J$, so that the dual sequences of $\{ P_{n} \}_{n \geq 0}$ and $\{ \tilde{P}_{n} \}_{n \geq 0}$ are related as the next Lemma points out.
\begin{lemma} \label{J(tilde un)} Let us consider a MPS $\{ P_{n} \}_{n \geq 0}$ and an operator $J$ of the form \eqref{operatorJ} fulfilling \eqref{deg-k1}-\eqref{deg-k2}. Thus,
\par \begin{center} $\quad J\left( \tilde{u}_{n} \right) = \lambda_{n+k}^{[k]} u_{n+k}$. \end{center}
\end{lemma}
\begin{proof}
We begin by applying the definition of dual sequence, together with prior definitions:
\begin{align*}
&\langle  \tilde{u}_{n} , \tilde{P}_{m}    \rangle = \delta_{n,m},\quad n, m \geq 0,\\
\Leftrightarrow & \langle  \tilde{u}_{n} , J\left( P_{m+k} \right)    \rangle =  \lambda_{m+k}^{[k]} \delta_{n,m},\quad n, m \geq 0,\\
\Leftrightarrow  & \langle J\left( \tilde{u}_{n} \right) , P_{m+k}    \rangle =  \lambda_{m+k}^{[k]} \delta_{n,m},\quad n, m \geq 0.
\end{align*}
Hence, $\langle J\left( \tilde{u}_{n} \right) , P_{m+k}  \rangle = 0,\; m \geq n+1$ and $\langle J\left( \tilde{u}_{n} \right) , P_{n+k}  \rangle = \lambda_{n+k}^{[k]} \neq 0$, by assumption on $J$.
Then, Lemma \ref{lema1} implies $J\left( \tilde{u}_{n} \right) = \displaystyle \sum_{\nu=0}^{n+k} \alpha_{n,\nu}u_{\nu}$, where $\alpha_{n,\nu} = \langle J\left( \tilde{u}_{n} \right), P_{\nu} \rangle = \langle  \tilde{u}_{n} , J \left( P_{\nu} \right) \rangle $, $0 \leq \nu \leq n+k$.
If $n=0$, we easily get $ J\left( \tilde{u}_{0} \right) = \alpha_{0,k} u_{k}$. Supposing $n \geq 1$, 
on one hand condition \eqref{deg-k2} yields  $\alpha_{n,\nu}=0$, if $0\leq \nu \leq k-1$, and on the other hand $\alpha_{n,\nu} =\langle  \tilde{u}_{n} ,  \lambda_{\nu}^{[k]} \tilde{P}_{\nu-k} \rangle =  \lambda_{\nu}^{[k]} \delta_{n, \nu-k}=0$, if $k \leq \nu \leq n+k-1$.
In conclusion, $J\left( \tilde{u}_{n} \right) = \alpha_{n,n+k}u_{n+k}$, with $\alpha_{n,n+k}= \lambda_{n+k}^{[k]}$, $ n \geq 0$.
\end{proof}

\begin{lemma} \label{tildeP=P} Let us consider a MPS $\{ P_{n} \}_{n \geq 0}$ and an operator $J$ of the form \eqref{operatorJ} such that \eqref{deg-k1}-\eqref{deg-k2} hold. Thus,
\begin{center} $\tilde{P}_{n}(x) = P_{n}(x),\;\; n \geq 0,\quad$ if and only if $\quad J\left( u_{n} \right) = \lambda_{n+k}^{[k]} u_{n+k}$. \end{center}
\end{lemma}
\begin{proof}
If we suppose that $\tilde{P}_{n}(x) = P_{n}(x),\;n \geq 0,$ then naturally, by means of Lemma \ref{J(tilde un)}, we get $J\left( u_{n} \right) = \lambda_{n+k}^{[k]} u_{n+k}$. Conversely, if $J\left( u_{n} \right) = \lambda_{n+k}^{[k]} u_{n+k}$ holds, then Lemma  \ref{J(tilde un)} yields $J\left( \tilde{u}_{n} \right)= J\left( u_{n} \right)$ which implies $\tilde{u}_{n}= u_{n},\; n \geq 0$, taking into account the attributes of $J$ and that for each $n$, we have $\langle \tilde{u}_{n} - u_{n}, B_{m}(x)  \rangle=0$, for all the elements of the PS $\{B_{m}\}_{m \geq 0}$, defined by $B_{m}(x)=J\left( \xi^{m+k} \right)(x)$.  
\end{proof}

In brief, given a non-negative integer $k$, when we consider the general problem of finding the MPSs $\textbf{P} = \left( P_{0}, P_{1}, \ldots \right)^{t}$ such that 
\begin{equation}
 \label{general-matrix-identity}\textbf{P} = \Lambda^{[k]}_{J} \; J \left( \textbf{P} \right),
 \end{equation}
where the matrix $\Lambda^{[k]}_{J}$ (with infinite dimensions) has the first $k$ columns filled with zeros and is defined as next,
$$ \Lambda^{[k]}_{J} =\left[\begin{array}{cccccccc}
\overbrace{ 0          \cdots  0  }^{ k\; \rm columns}      & \left(\lambda_{k}^{[k]}\right)^{-1}& 0 & \cdots &  &  & &\cdots \\
 \vdots \quad\;\;\,    \vdots & 0 &  \left(\lambda_{1+k}^{[k]}\right)^{-1} & 0 & \cdots &  & &\cdots \\
    0        \quad\;\;           0           & 0 & 0 &\ddots & 0 &  & &\cdots \\
                                     & 0 & 0 & 0 &  \left(\lambda_{n+k}^{[k]}\right)^{-1}& 0 & &\cdots \\
 \vdots  \quad\;\;\,              \vdots & 0 & 0 & 0 & 0 & \ddots &&  \ddots  \\
 \end{array}\right]  $$
a common technique can be pursuit by taking into consideration that  \eqref{general-matrix-identity} is equivalent to a shift of order $k$ in the dual sequence $\textbf{u}= \left( u_{0}, u_{1}, \ldots \right)^{t}$ as established in Lemma \ref{tildeP=P}, or through the similar identity:
\begin{equation}
J\left( \textbf{u} \right) = \left[\begin{array}{cccccccc}
\overbrace{ 0          \cdots  0  }^{ k\; \rm columns}      & \lambda_{k}^{[k]}& 0 & \cdots &  &  & &\cdots \\
 \vdots \quad\;\;\,    \vdots & 0 &  \lambda_{1+k}^{[k]} & 0 & \cdots &  & &\cdots \\
    0        \quad\;\;           0           & 0 & 0 &\ddots & 0 &  & &\cdots \\
                                     & 0 & 0 & 0 &  \lambda_{n+k}^{[k]}& 0 & &\cdots \\
 \vdots  \quad\;\;\,              \vdots & 0 & 0 & 0 & 0 & \ddots &&  \ddots  \\
 \end{array}\right]  \; \textbf{u}.
\end{equation}

\section{Differential relations: general Appell behavior and a review on classical orthogonal polynomials}\label{sec:orto} 
To illustrate the application of the theoretical exposition of the previous sections in finding the MPSs $\{P_{n} \}_{n \geq 0}$ fulfilling $\tilde{P}_{n}(x) = P_{n}(x),\;n \geq 0,$ with  $\tilde{P}_{n}(x)$ defined by \eqref{JPn}, we will now 
consider an operator $J$ of second order with three terms ($a_{2}(x)$ non-trivial), as follows
\begin{equation}
J = a_{0}(x)I + a_{1}(x)D+\frac{a_{2}(x)}{2}D^{2},
\end{equation}
and moreover we will suppose that the given MPS $\{P_{n} \}_{n \geq 0}$ is orthogonal, which can be translated in terms of the elements of the dual sequence through the next recurrence of Theorem \ref{regular}:
\begin{equation}\label{xun}
xu_{n}=u_{n-1}+\beta_{n}u_{n}+\gamma_{n+1}u_{n+1}, \quad \gamma_{n+1} \neq
  0,\;\;\; n \geq 0,\quad  u_{-1}=0.
\end{equation}
A further important technical aspect is given by identity \eqref{J(fu)} which asserts in this case that
\begin{equation}\label{J(xu)}
J(xu)= xJ(u)-a_{1}(x)u +D\left(a_{2}(x) u\right),
\end{equation}
where, in particular, 
\begin{equation*}
J(u)= a_{0}(x)u - D\left(a_{1}(x)u \right)+\frac{1}{2}D^{2}\left(a_{2}(x)u\right)\; \textrm{and}\; J^{(1)}(u)=a_{1}(x)u-D\left(a_{2}(x)u \right),
\end{equation*}
as identities \eqref{transpose-J} and \eqref{J^(m)} establish.

\subsection{$J$ is an isomorphism}
If we assume that $J$ is an isomorphism, we may look at Lemma \ref{J-descending} with $k=0$ and hence we are requiring that $\deg \left( a_{\nu}(x) \right) \leq \nu,\; \nu=0,1,2$ and
\begin{equation}
\lambda_{n}^{[0]}= a_{0}^{[0]}+n a_{1}^{[1]}+ \frac{n(n-1)}{2}a_{2}^{[2]} \neq 0,\quad n \geq 0.
\end{equation}
The identity $\tilde{P}_{n}(x) = P_{n}(x),\;n \geq 0,$ can be read as the following differential relation, which put us on the path of classical solutions (cf. \cite{Bochner, euler}).
\begin{equation}
a_{2}(x) P^{\prime\prime}_{n}(x)+2a_{1}(x)P^{\prime}_{n}(x) + 2 \left( a_{0}(x)-\lambda_{n}^{[0]}  \right)P_{n}(x)=0,\;  n \geq 0.
\end{equation}
If we apply $J$ (transpose) to identity \eqref{xun} and introduce \eqref{J(xu)} and Lemma \ref{tildeP=P}, we obtain the following relation where $\lambda_{-1}^{[0]}=0$.
\begin{align}
\label{eq-orto-k0-n}&\left(\lambda_{n}^{[0]}\left( x-\beta_{n}\right) -a_{1}(x)\right)u_{n} = \lambda_{n-1}^{[0]}u_{n-1}+\gamma_{n+1}\lambda_{n+1}^{[0]}u_{n+1} - D\left(a_{2}(x) u_{n} \right).
\end{align}
Taking \eqref{eq-orto-k0-n} with $n=0$, we get:
\begin{align*}
&\left(\lambda_{0}^{[0]}\left( x-\beta_{0}\right) -a_{1}(x)\right)u_{0} = \gamma_{1}\lambda_{1}^{[0]}u_{1} - D\left(a_{2}(x) u_{0} \right).
\end{align*}
The orthogonality hypothesis allows us to replace $u_1$ by $\left( \gamma_{1}\right)^{-1} \left( x-\beta_{0}\right)u_{0}$ (cf. item e) of Theorem \ref{regular}), yielding
\begin{align*}
&D\left(a_{2}(x) u_{0} \right)  + \left( \left(\lambda_{0}^{[0]}-\lambda_{1}^{[0]} \right)\left( x-\beta_{0}\right) -a_{1}(x)\right)u_{0} = 0,
\end{align*}
and since $\lambda_{0}^{[0]}= a_{0}^{[0]}$ and $\lambda_{1}^{[0]}= a_{0}^{[0]}+a_{1}^{[1]}$, we actually have
\begin{align}\label{functional-identity-k=0}
&D\left(a_{2}(x) u_{0} \right)  + \left(-2a_{1}^{[1]}x+a_{1}^{[1]}\beta_{0}-a_{0}^{[1]} \right)u_{0} = 0.
\end{align}
In addition, the identity $J\left(P_{n}\right)(x) = \lambda_{n}^{[0]}P_{n}(x)$ provides with $n=1$ the following information.
\begin{equation}\label{beta0-k0}
\beta_{0}=-\frac{a_{0}^{[1]}}{a_{1}^{[1]}}.
\end{equation}
Before further analysis, let us argue that $a_{1}^{[1]} \neq 0$, by supposing the opposite which implies that  $D\left(a_{2}(x) u_{0} \right)  = a_{0}^{[1]} u_{0}.$
\newline If $a_{0}^{[1]} = 0$, then $D\left(a_{2}(x) u_{0} \right) =0$ which will imply, by Lemma  \ref{phi-u=0}, $a_{2}(x) =0$. If $a_{0}^{[1]} \neq 0$, then 
\begin{align*}
&-\langle  a_{2}(x) u_{0} , P^{\prime}_{n} \rangle = \langle  D\left(a_{2}(x) u_{0} \right) , P_{n} \rangle = \langle  a_{0}^{[1]} u_{0} , P_{n} \rangle = a_{0}^{[1]}\delta_{0,n},\quad n \geq 0.
\end{align*}
In particular, when $n=0$ we obtain $\langle  a_{2}(x) u_{0} , 0 \rangle = -a_{0}^{[1]} \neq 0$ which is contradictory.

As a consequence, the functional identity \eqref{functional-identity-k=0} corresponds to a classical family of polynomial sequences as recalled before in \eqref{classical-equation}, with initial (non-normalized) polynomials:
\begin{align}
\label{phi-k0} &\phi(x) = a_{2}(x)=a_{2}^{[2]}x^2+a_{1}^{[2]}x + a_{0}^{[2]};\\
\label{psi-k0} &\psi(x) = -2a_{1}(x)=-2a_{1}^{[1]}x-2a_{0}^{[1]}, \quad a_{1}^{[1]} \neq 0,
\end{align}
though we must add the admissible condition $\psi^{\prime}-\frac{1}{2}\phi^{\prime\prime}n \neq 0,\: n \geq 1$, or more precisely
\begin{equation}\label{bessel-admissible}
2a_{1}^{[1]} + a_{2}^{[2]} n \neq 0,\, n \geq 0,
\end{equation}
so that \eqref{functional-identity-k=0} assures the recursive computation of the moments of $u_{0}$ \cite{theoriealgebrique}.
Depending on the degree of $a_2(x)$ and upon different affine transformations we shall be guided though the four classes of classical polynomial sequences, as expected (cf. \cite{variations}). 
\begin{itemize}
\item[A)] First case: $\deg\left(a_{2}(x) \right) = 2$.
In order to have a normalized $\phi(x)$, let us replace \eqref{phi-k0}-\eqref{psi-k0} by
\begin{align}
\label{phi-k0-n} &\phi_{1}(x) =x^2+\frac{a_{1}^{[2]}}{a_{2}^{[2]}}x + \frac{a_{0}^{[2]}}{a_{2}^{[2]}};\\
\label{psi-k0-n} &\psi_{1}(x) = -2\frac{a_{1}^{[1]}}{a_{2}^{[2]}}x+\frac{1}{a_{2}^{[2]}}\left( a_{1}^{[1]}\beta_{0}-a_{0}^{[1]}\right), \quad a_{1}^{[1]} \neq 0.
\end{align}
Let us rearrange $\phi_{1}(x)$ as $\phi_{1}(x)=(x-d)^2-\mu$, where 
$$\mu=-\frac{a_{0}^{[2]}}{a_{2}^{[2]}}+ \left( \frac{a_{1}^{[2]}}{2a_{2}^{[2]}} \right)^2,\quad d= - \frac{a_{1}^{[2]}}{2a_{2}^{[2]}}.$$ 
$\bullet$ Let us suppose that $\mu=0$ and thus $\phi_{1}(x)=(x-d)^2$.
Applying the affine transformation $x \mapsto x+d$ and following \eqref{phi-psi-affine}, we conclude that  the form $\tilde{u}_{0}= \tau_{-d} u_{0}$ fulfils $D\left( \tilde{\phi_{1}} \tilde{u}_{0} \right)  + \tilde{\psi_{1}} \tilde{u}_{0}= 0$, with
$$\tilde{\phi_{1}}(x)=x^2,\quad \tilde{\psi_{1}}(x)= -2\frac{a_{1}^{[1]}}{a_{2}^{[2]}}\left( x+d \right) + \frac{a_{1}^{[1]}}{a_{2}^{[2]}}\beta_{0}- \frac{a_{0}^{[1]}}{a_{2}^{[2]}}.$$
Defining $\alpha= \dfrac{a_{1}^{[1]}}{a_{2}^{[2]}}$ we get
$ \tilde{\psi_{1}}(x)= -2\alpha \left( x+d \right) + \alpha \beta_{0} +\dfrac{a_{1}^{[1]}}{a_{2}^{[2]}}\left( -\dfrac{a_{0}^{[1]}}{a_{1}^{[1]}} \right)$
and in view of \eqref{beta0-k0}, we can write $ \tilde{\psi_{1}}(x)= -2\alpha \left( x+d -\beta_{0} \right)$.
The recurrence coefficients corresponding to the regular form $\tilde{u}_{0}$ can be easily obtained by  \eqref{shifted-coef}, in particular, we have $\tilde{\beta}_{0}=\beta_{0}-d$, allowing us to notice that $\tilde{\phi_{1}}(x)=x^2$ and $ \tilde{\psi_{1}}(x)= -2\alpha \left( x-\tilde{\beta}_{0} \right),$ or in other words, $\tilde{u}_{0}$ is a Bessel form with parameter $\alpha= \dfrac{a_{1}^{[1]}}{a_{2}^{[2]}}$.

\noindent $\bullet$ Let us suppose that $\mu \neq 0$ and thus $\phi_{1}(x)=(x-d)^2-\mu$.
Applying the affine transformation $x \mapsto \sqrt{\mu}\,x+d$ and following \eqref{phi-psi-affine}, we conclude that  the form $\tilde{u}_{0} =\left( h_{\left(\sqrt{\mu}\right)^{-1}}\circ \tau_{-d} \right) u_{0}$ fulfils $D\left( \tilde{\phi_{1}} \tilde{u}_{0} \right)  + \tilde{\psi_{1}} \tilde{u}_{0}= 0$, with
$$\tilde{\phi_{1}}(x)=x^2-1,\quad \tilde{\psi_{1}}(x)= -\frac{2a_{1}^{[1]}}{a_{2}^{[2]}} \sqrt{\mu}\,x- \frac{2a_{1}^{[1]}}{a_{2}^{[2]}}d+ \frac{2a_{1}^{[1]}}{a_{2}^{[2]}}\beta_{0}.$$
Rewriting this last polynomial, we see a Jacobi form, because
$\tilde{\phi_{1}}(x)=x^2-1$ and $\tilde{\psi_{1}}(x)= -\left( \alpha + \beta+2 \right)x+\alpha -\beta$
with
\newline $\; \alpha = \dfrac{a_{1}^{[1]}}{a_{2}^{[2]}}\left(  \sqrt{\mu}-d-\dfrac{a_{0}^{[1]}}{a_{1}^{[1]}} \right)-1$ and $\beta = \dfrac{a_{1}^{[1]}}{a_{2}^{[2]}}\left(  \sqrt{\mu}+d+\dfrac{a_{0}^{[1]}}{a_{1}^{[1]}} \right)-1$.

\item[B)] Second case: $\deg\left(a_{2}(x) \right) = 1$. 
In order to have a monic $\phi(x)$, let us replace \eqref{phi-k0}-\eqref{psi-k0} by
\begin{align}
\label{phi-k0-case2} &\phi_{1}(x) =x + \frac{a_{0}^{[2]}}{a_{1}^{[2]}};\\
\label{psi-k0-case2} &\psi_{1}(x) = -\frac{2a_{1}^{[1]}}{a_{1}^{[2]}}x-\frac{2a_{0}^{[1]}}{a_{1}^{[2]}}, \quad a_{1}^{[2]} \neq 0.
\end{align}
In particular, condition \eqref{bessel-admissible} simply indicates $-2a_{1}^{[1]}\neq 0$.
Applying the affine transformation $x \mapsto Ax+B$, with $A=-\frac{a_{1}^{[2]}}{2a_{1}^{[1]}}$ and $B=-\frac{a_{0}^{[2]}}{a_{1}^{[2]}}$, and following \eqref{phi-psi-affine}, we conclude that  the form $\tilde{u}_{0} =\left( h_{A^{-1}}\circ \tau_{-B} \right) u_{0}$ fulfils $D\left( \tilde{\phi_{1}} \tilde{u}_{0} \right)  + \tilde{\psi_{1}} \tilde{u}_{0}= 0$, with
$$\tilde{\phi_{1}}(x)=x,\quad \tilde{\psi_{1}}(x)=x+\frac{2a_{1}^{[1]}a_{0}^{[2]}}{a_{1}^{[2]}a_{1}^{[2]}}-\frac{2a_{0}^{[1]}}{a_{1}^{[2]}},$$
corresponding to a Laguerre form of parameter $\alpha = -1+\frac{2}{a_{1}^{[2]}}\left( a_{0}^{[1]} -  \frac{a_{0}^{[2]}a_{1}^{[1]}}{a_{1}^{[2]}}\right)$.

\item[C)] Third case: $\deg\left(a_{2}(x) \right) = 0$. 
In a similar way, we have in this case $\phi_{1}(x) =1$ and $\psi_{1}(x) = -\dfrac{2a_{1}^{[1]}}{a_{0}^{[2]}} x - \dfrac{2a_{0}^{[1]}}{a_{0}^{[2]}}$, and again condition \eqref{bessel-admissible} points to the already known condition $-2a_{1}^{[1]}\neq 0$. Applying the affine transformation $x \mapsto Ax+B$, with $A=\sqrt{-\frac{a_{0}^{[2]}}{a_{1}^{[1]}}}$ and $B=-\frac{a_{0}^{[1]}}{a_{1}^{[1]}}=\beta_{0}$, and following \eqref{phi-psi-affine}, we conclude that  the form $\tilde{u}_{0} =\left( h_{A^{-1}}\circ \tau_{-B} \right) u_{0}$ fulfils $D\left( \tilde{u}_{0} \right)  + 2x \tilde{u}_{0}= 0$, meaning that it is a Hermite form.
\end{itemize}

\subsection{$J$ is a lowering operator of order one}
If we assume that $J$ is a lowering operator of order one and then $k=1$, Lemma \ref{J-descending} clarifies that $\deg \left( a_{\nu}(x) \right) \leq \nu-1,\; \nu=0,1,2$ and
\begin{equation}\label{lambda-k1}
\lambda_{n+1}^{[1]}= (n+1)a_{0}^{[1]}+ \frac{n(n+1)}{2} a_{1}^{[2]} \neq 0,\quad n \geq 0.
\end{equation}
Also, the identity $\tilde{P}_{n}(x) = P_{n}(x),\;n \geq 0,$ can be read as the following differential relation.
\begin{equation}
\frac{1}{2}\left(a_{0}^{[2]}+a_{1}^{[2]}x \right) P^{\prime\prime}_{n+1}(x)+a_{0}^{[1]}P^{\prime}_{n+1}(x) =\lambda_{n+1}^{[1]} P_{n}(x),\;  n \geq 0.
\end{equation}
In this case, Lemma \ref{tildeP=P} asserts that  $\tilde{P}_{n}(x) = P_{n}(x),\;n \geq 0,$ if and only if $J\left(  u_{n} \right) = \lambda_{n+1}^{[1]} u_{n+1}$, $n \geq 0$. Therefore, applying $J$ (transpose) to identity \eqref{xun} and in view of \eqref{J(xu)} we get the following, for $ n \geq 0$, with $\lambda_{0}^{[1]} =0$.
\begin{align}
\label{eq-orto-k1-n}& \lambda_{n+1}^{[1]}\left( x-\beta_{n}   \right)u_{n+1}-\left( a_{1}(x)+  \lambda_{n}^{[1]} \right) u_{n} - \gamma_{n+1}\lambda_{n+2}^{[1]}u_{n+2} + D\left(a_{2}(x) u_{n} \right) = 0.
\end{align}
Taking \eqref{eq-orto-k1-n} with $n=0$, we obtain:
$$\lambda_{1}^{[1]} \left( x-\beta_{0} \right) u_{1} - a_{1}(x)u_{0} - \gamma_{1}\lambda_{2}^{[1]} u_{2} + D \left(a_{2}(x) u_{0} \right) = 0.$$
Applying the orthogonality through the information $u_{1}= \left( \gamma_{1} \right)^{-1}P_{1}(x) u_{0}$ and $u_{2}= \left( \gamma_{1}  \gamma_{2} \right)^{-1}P_{2}(x) u_{0}$, we finally obtain:
\begin{equation}\label{functionalidentityk=1}
D \left( a_{2}(x) u_{0} \right) + \left( \frac{\lambda_{1}^{[1]}}{\gamma_{1}} \left(x-\beta_{0} \right)^{2} -\frac{\lambda_{2}^{[1]}}{\gamma_{2}} P_{2}(x) - a_{1}(x) \right)u_{0}=0.
\end{equation}
Let us proceed by examining precedent identities, now adapted to the case under study.
To do so, we notate by $\psi(x)$ the polynomial $\frac{\lambda_{1}^{[1]}}{\gamma_{1}} \left(x-\beta_{0} \right)^{2} -\frac{\lambda_{2}^{[1]}}{\gamma_{2}} P_{2}(x) - a_{1}(x) $, and notice that \eqref{functionalidentityk=1} is equivalent to
\begin{equation}\label{Basic funct ident of k=1}
a_{2}(x)D\left( u_{0} \right) = -\left( a_{1}^{[2]} +\psi(x) \right)u_{0}.
\end{equation}
Moreover, the first identity of Lemma \ref{tildeP=P}:
\begin{equation*}
J(u_{0})= \lambda_{1}^{[1]}u_{1} \Leftrightarrow - D\left(a_{0}^{[1]} u_{0} \right)+\frac{1}{2}D^{2}\left(a_{2}(x)u_{0}\right) = \frac{\lambda_{1}^{[1]}}{\gamma_{1}}P_{1}(x)u_{0},
\end{equation*}
yields $D^{2}\left(a_{2}(x)u_{0}\right) = \frac{2\lambda_{1}^{[1]}}{\gamma_{1}}P_{1}(x)u_{0}+2a_{0}^{[1]} D\left(u_{0} \right)$, whereas the differentiation of \eqref{functionalidentityk=1} provides
$D^{2}\left( a_{2}(x) u_{0}\right) + D\left( \psi(x) u_{0}\right) = 0$.
The consequent elimination of $D^{2}\left( a_{2}(x) u_{0}\right) $ allows us to conclude 
\begin{equation*}
 \left( \psi(x)+2a_{0}^{[1]} \right)D\left(u_{0}\right) + \left(   \psi^{\prime}(x)+\frac{2\lambda_{1}^{[1]}}{\gamma_{1}}P_{1}(x)\right) u_{0} = 0.
\end{equation*}
Multiplying this last identity by the polynomial $a_{2}(x)= a_{0}^{[2]}+a_{1}^{[2]} x$ (left multiplication) and using both \eqref{Basic funct ident of k=1} and Lemma \ref{phi-u=0}, we get
$$-\left( \psi(x)+a_{1}^{[2]} \right) \left( \psi(x)+2a_{0}^{[1]}  \right) +\left( \psi^{\prime}(x) +  \frac{2\lambda_{1}^{[1]}}{\gamma_{1}}P_{1}(x) \right)a_{2}(x) =0.$$ 
An immediate consequence of this polynomial identity regards the degree of $\psi(x)$, since we conclude that
$\frac{\lambda_{1}^{[1]}}{\gamma_{1}}- \frac{\lambda_{2}^{[1]}}{\gamma_{2}}=0$ and henceforth $\deg \psi \leq 1$.
\newline The deduced condition $\frac{\lambda_{1}^{[1]}}{\gamma_{1}}- \frac{\lambda_{2}^{[1]}}{\gamma_{2}} =0$ justifies the additional condition: 
\begin{equation}\label{condition1-k1}
a_{1}^{[2]} = a_{0}^{[1]} \frac{\gamma_{2}-2\gamma_{1}}{\gamma_{1}} , \; \textrm{since}\; \lambda_{1}^{[1]} = a_{0}^{[1]} \neq 0,\; \lambda_{2}^{[1]} = a_{1}^{[2]} + 2a_{0}^{[1]} \neq 0, 
\end{equation}
under which we obtain $D \left( \phi(x) u_{0} \right) + \psi(x)u_{0}=0$, where
\begin{center} 
$\phi(x)= a_{2}(x) = a_{1}^{[2]}x + a_{0}^{[2]} $ and 
$\psi(x)= - \dfrac{a_{0}^{[1]}}{\gamma_{1}} \left( \beta_{0}-\beta_{1} \right) \left( x-\beta_{0} \right)$.\end{center}
This set corresponds to a classical form if and only if $\psi^{\prime}(x)=\frac{a_{0}^{[1]}}{\gamma_{1}} \left( \beta_{0}-\beta_{1} \right) \neq 0$. It is easy to notice that $\deg \left( \psi \right) = 0$ yields $\psi(x)=0$ and $a_{2}(x)=0$. Dividing the latter functional identity by $a_{1}^{[2]}$ and using \eqref{condition1-k1}, where obviously $\gamma_{2}-2\gamma_{1} \neq 0$ by assumption on $a_{1}^{[2]}$, we get
$$ D\left( \left(x+ \frac{a_{0}^{[2]}}{a_{1}^{[2]}}  \right) u_{0}\right) - \frac{\beta_{0}-\beta_{1}}{\gamma_{2}-2\gamma_{1}} \left( x-\beta_{0}\right)u_{0}=0.$$
Applying the affine transformation $x \mapsto Ax+B$, with $A=-\frac{\gamma_{2}-2\gamma_{1}}{\beta_{0}-\beta_{1}}$ and $B=-\frac{a_{0}^{[2]}}{a_{1}^{[2]}}$, and following \eqref{phi-psi-affine}, we conclude that  the form $\tilde{u}_{0}=\left(h_{A^{-1}}\circ \tau_{-B} \right)u_{0}$ fulfils 
$$D\left(x \tilde{u}_{0}\right) +\left( x + \frac{\beta_{0}-\beta_{1}}{\gamma_{2}-2\gamma_{1}}\left(  \frac{a_{0}^{[2]}}{a_{1}^{[2]}}  +\beta_{0}\right)   \right) \tilde{u}_{0}=0.$$
This type of functional identity is known in the literature to characterize a Laguerre form, with parameter $\alpha  \notin \mathbb{Z}^{-}$ so that
$-1-\alpha = \frac{\beta_{0}-\beta_{1}}{\gamma_{2}-2\gamma_{1}}\left(  \frac{a_{0}^{[2]}}{a_{1}^{[2]}}  +\beta_{0}\right) $.
Let us remark that in view of \eqref{condition1-k1}, the dilation factor $A$ can also be written as follows.
\begin{equation}\label{A}
A= \frac{\gamma_{1}a_{1}^{[2]}}{\left( \beta_{1}-\beta_{0}  \right)a_{0}^{[1]} }
\end{equation}
Taking into account \eqref{shifted-coef} and recalling the recurrence coefficients of a Laguerre form: $\tilde{\beta}_{n}=2n+\alpha+1$ , $\tilde{\gamma}_{n+1}=(n+1)(n+\alpha+1)$, we conclude that
$$ (\alpha+1)A = \beta_{0}+\frac{a_{0}^{[2]}}{a_{1}^{[2]}} \;,\quad  (\alpha+3)A = \beta_{1}+\frac{a_{0}^{[2]}}{a_{1}^{[2]}}\; ,\quad (\alpha+1)A^2 = \gamma_{1}\,; $$
allowing us to obtain $2A=\beta_{1}-\beta_{0} $ and inserting \eqref{A} on $ (\alpha+1)A = \frac{\gamma_{1}}{A}$ we get
\begin{equation}
\alpha = \frac{2a_{0}^{[1]}}{a_{1}^{[2]}}-1.
\end{equation}
It is worth noting that $\frac{2a_{0}^{[1]}}{a_{1}^{[2]}}-1 \notin \mathbb{Z}^{-}$ due to \eqref{lambda-k1}.
\begin{remark}
In order to compare the above conclusions with the treatment of the same problem given in \cite{QuadraticAppell} and \cite{MesquitaPMH}
 for the operator $\Lambda = \sigma_{0} D+ \sigma_{1} DxD$, we must first write this $\Lambda$ in the $J$ format. A few calculations yield
 $$a_{0}^{[2]}=0, \; a_{1}^{[2]}=2 \sigma_{1},\; a_{0}^{[1]}=\sigma_{0} + \sigma_{1}$$
so that $\lambda_{n+1}^{[1]}= (n+1)\left( \sigma_{0}+(n+1)\sigma_{1})\right) \neq 0$ and $\alpha= \dfrac{\sigma_{0}}{\sigma_{1}}$.
\end{remark}

\subsection{$J$ is a lowering operator of order two}
Considering $k=2$, by Lemma \ref{J-descending} we have $\deg \left( a_{\nu}(x) \right) \leq \nu-2,\; \nu=0,1,2$ and
\begin{equation}
\lambda_{n+2}^{[2]}  = \frac{(n+1)(n+2)}{2} a_{0}^{[2]}\neq 0,\quad n \geq 0.
\end{equation}
The identity $\tilde{P}_{n}(x) = P_{n}(x),\;n \geq 0,$ can be read as follows.
\begin{equation}
\frac{1}{2} a_{0}^{[2]} P^{\prime\prime}_{n+2}(x)=\lambda_{n+2}^{[2]} P_{n}(x),\;  n \geq 0.
\end{equation}
Applying the same procedure as before, we get from \eqref{xun} and  \eqref{J(xu)} the following, for $ n \geq 0$, with $\lambda_{1}^{[2]} =0$.
\begin{align}
\label{eq-orto-k2-n}& \lambda_{n+2}^{[2]}\left( x-\beta_{n}   \right)u_{n+2}- \lambda_{n+1}^{[2]}u_{n+1} - \gamma_{n+1}\lambda_{n+3}^{[2]}u_{n+3} + D\left(a_{2}(x) u_{n} \right) = 0.
\end{align}
Taking $n=0$ and considering the orthogonality conditions, we obtain
\begin{equation} \label{functionalidentityk=2}
D\left(a_{0}^{[2]} u_{0} \right) +  \psi(x) u_{0}=0\end{equation}
with $\psi(x)= a_{0}^{[2]}\left( x-\beta_{0}\right) \left( \gamma_{1}\gamma_{2} \right)^{-1}P_{2}(x) -   3a_{0}^{[2]}\left( \gamma_{2}\gamma_{3} \right)^{-1}P_{3}(x)$.
\newline Being $k=2$, we have
\begin{equation*}
J(u_{0})= \lambda_{2}^{[2]}u_{2} \Leftrightarrow \frac{1}{2}D^{2}\left(a_{0}^{[2]}u_{0}\right) = \frac{\lambda_{2}^{[2]}}{\gamma_{1}\gamma_{2}}P_{2}u_{0},
\end{equation*}
whereas the differentiation of \eqref{functionalidentityk=2} provides
$D^{2}\left( a_{0}^{[2]} u_{0}\right) + D\left( \psi(x) u_{0}\right) = 0$.
The elimination of $D^{2}\left( a_{2}(x) u_{0}\right) $, allows us to obtain firstly
$$  \frac{2\lambda_{2}^{[2]}}{\gamma_{1}\gamma_{2}}P_{2}u_{0}+ \psi^{\prime}(x) u_{0} + \psi(x)D\left(u_{0}\right)=0,$$
and secondly using \eqref{functionalidentityk=2} to replace $D\left(u_{0}\right)$ we get :
\begin{equation}
\left( \frac{2\lambda_{2}^{[2]}}{\gamma_{1}\gamma_{2}}P_{2} + \psi^{\prime}(x) - \frac{1}{a_{0}^{[2]}}\psi^{2}(x) \right)u_{0}= 0,
\end{equation}
which implies $\frac{2\lambda_{2}^{[2]}}{\gamma_{1}\gamma_{2}}P_{2} + \psi^{\prime}(x) - \frac{1}{a_{0}^{[2]}}\psi^{2}(x) = 0$, by 
Lemma \ref{phi-u=0}. In view of this last polynomial identity we conclude that the only admissible solution is tied to the case where $\deg\left( \psi(x) \right) = 1$. Since $\phi(x)=a_{0}^{[2]}$, the final solution is an affine transformation of the Hermite form which was a already known solution for the reason that the Hermite sequence is the single Appell MOPS \cite{euler}.

\section{An application to the two-orthogonal polynomial sequences}\label{sec:twoorto}

Focusing now on the two-orthogonality, also known as a particular case of the step-line multiple orthogonal polynomials \cite{van-assche}, we concisely recall that a  two-orthogonal MPS  can be recursively computed by means of three constant sequences $\{\beta_{n}\}_{n \geq 0}$, $\{\alpha_{n}\}_{n \geq 1}$ and $\{\gamma_{n}\}_{n \geq 1}$, with $\gamma_{n+1}  \neq 0,\; n \geq 0$ \cite{Maroni-NumAlg, Douak-two-Laguerre} as indicated in \eqref{ic-2orto}-\eqref{rr-2orto}. These are the structure coefficients of the two-orthogonal polynomial sequences.
\begin{align}
\label{ic-2orto}  & P_{0}(x)=1,\:\: P_{1}(x)=x-\beta_{0},\;\:  P_{2}(x)=\left( x-\beta_{1} \right) P_{1}(x) - \alpha_{1}; \\
\label{rr-2orto}  & P_{n+3}(x) = \left( x-\beta_{n+2} \right) P_{n+2}(x) - \alpha_{n+2} P_{n+1}(x) - \gamma_{n+1} P_{n}(x).
\end{align}
In a similar manner to the orthogonal case, its dual sequence fulfils a recurrence relation based on those structure coefficients, as follows.
\begin{equation} \label{functional-2orto}
x u_{n} = u_{n-1} + \beta_{n}u_{n} + \alpha_{n+1} u_{n+1} +  \gamma_{n+1} u_{n+2}, \;  \textrm{with} \; n \geq 0, \,\; u_{-1} = 0.
\end{equation}
Moreover, all elements of the dual sequence can be written in terms of the pair $(u_{0}, u_{1})$. In particular, we have \cite{Maroni-NumAlg} (p. 307)
\begin{align}
\label{Eq-u2} u_{2} &= E_{1}(x) u_{0} + A_{0}(x) u_{1}\\
\label{Eq-u3} u_{3} &= B_{1}(x) u_{0} + F_{1}(x) u_{1}\\\
\label{Eq-u4} u_{4} &= E_{2}(x) u_{0} + A_{1}(x) u_{1}\
\end{align}
where
\begin{align*}
& E_{1}(x)= \frac{1}{\gamma_{1}} \left( x-\beta_{0}\right); \quad A_{0}(x)= - \frac{\alpha_{1}}{\gamma_{1}}; \quad 
B_{1}(x) = - \frac{\alpha_{2}}{\gamma_{1}\gamma_{2}} \left( x-\beta_{0}\right) - \frac{1}{\gamma_{2}} ; \\
& F_{1}(x) = \frac{1}{\gamma_{2}} \left( x-\beta_{1} + \frac{\alpha_{2}\alpha_{1}}{\gamma_{1}} \right); \quad E_{2}(x) = \frac{1}{\gamma_{3}} \big( \left( x-\beta_{2}\right) E_{1}(x) -\alpha_{3} B_{1}(x) \big); \\
& A_{1}(x) = -\frac{1}{\gamma_{3}} \left( \alpha_{3}F_{1}(x)+1+\frac{\alpha_{1}}{\gamma_{1}}\left( x-\beta_{2} \right)\right).  
\end{align*}
Let us keep considering an operator $J$ with only three terms, as follows
\begin{equation}\label{threeterms}
J = a_{0}(x)I + a_{1}(x)D+\frac{a_{2}(x)}{2}D^{2},
\end{equation}
and a two-orthogonal sequence $\{P_{n} \}_{n \geq 0}$ fulfilling $\tilde{P}_{n}(x) = P_{n}(x),\;n \geq 0,$ with  $\tilde{P}_{n}(x)$ defined by \eqref{JPn} with $k=1$ meaning that $J$ is a lowering operator of order one. The application of the same type of reasoning to the above new data regarding the two-orthogonality allows us now to state the following result.

\begin{proposition}
Let us consider a two-orthogonal MPS $\{P_{n} \}_{n \geq 0}$ and a lowering operator (of order one) $J$ of the form \eqref{threeterms}, so that $J\left( P_{n+1}(x) \right) =\lambda_{n+1}^{[1]} P_{n}(x), \; n \geq 0$, being $\lambda_{n+1}^{[1]}$ defined by \eqref{lambda-k1}. Then, the $2$-dimensional functional $U=\left( u_{0}, u_{1} \right)^{T}$ fulfils the vectorial relation $ D \left( \Phi U \right) +\Psi U = 0$, declaring that $\{P_{n}\}_{n \geq 0}$ must be two-classical (in Hahn's sense as in \cite{d-classical}), with
\newline
\begin{tabular}{cc}
$\quad\quad$ $\Psi(x)=\left[\begin{array}{cc}0 & 1 \\ \dfrac{2}{\gamma_{1}} \left( x-\beta_{0}\right) & -2\dfrac{\alpha_{1}}{\gamma_{1}} \end{array}\right], $
 & 
$\Phi(x)= \left[\begin{array}{cc} \phi_{1,1}(x) & \phi_{1,2}(x) \\  \phi_{2,1}(x) & \phi_{2,2}(x)\end{array}\right]$,
 \end{tabular}
\newline where the polynomials $\phi_{1,1}(x),\; \phi_{1,2}(x),\; \phi_{2,1}(x)$ and  $\phi_{2,2}(x)$ are defined by identities \eqref{11coeffmatrixEq-F} - \eqref{22coeffmatrixEq-F}, respectively.
\end{proposition} 

\begin{proof}
Let us begin by noticing that by Lemma  \ref{tildeP=P}  $J$ is in fact expressed by $J= a_{1}(x)D + \frac{a_{2}(x)}{2}D^{2}$ with $a_{1}(x)=a_{0}^{[1]}$ and $a_{2}(x)= a_{0}^{[2]} + a_{1}^{[2]} x$.
Applying $J$ on identity \eqref{functional-2orto} and taking into account \eqref{J(xu)} and $J\left( u_{n} \right) = \lambda_{n+1}^{[1]} u_{n+1}$, $n \geq 0$, we obtain the following relation for $ n \geq 0$ with $\lambda_{0}^{[1]} =0$:
\begin{equation*}
 \lambda_{n+1}^{[1]} \left( x- \beta_{n} \right) u_{n+1} -a_{1}(x)u_{n} + D\left(a_{2}(x) u_{n}\right)
=  \lambda_{n}^{[1]}u_{n}+ \alpha_{n+1}\lambda_{n+2}^{[1]} u_{n+2} + \gamma_{n+1} \lambda_{n+3}^{[1]} u_{n+3}. 
\end{equation*}
Taking $n=0$ and $n=1$ and differentiating we get the two following functional relations.
\begin{align}\label{n=0andn=1}
&D\left(  \lambda_{1}^{[1]} \left( x- \beta_{0} \right) u_{1} -a_{1}(x)u_{0} \right)+ D^{2}\left(a_{2}(x) u_{0}\right)
+D\left( -\alpha_{1}\lambda_{2}^{[1]} u_{2} - \gamma_{1} \lambda_{3}^{[1]} u_{3} \right) =0\\
\nonumber &D\left( \lambda_{2}^{[1]} \left( x- \beta_{1} \right) u_{2} -a_{1}(x)u_{1} \right) + D^2\left(a_{2}(x) u_{1}\right)
+D\left( - \lambda_{1}^{[1]}u_{1}- \alpha_{2}\lambda_{3}^{[1]} u_{3} - \gamma_{2} \lambda_{4}^{[1]} u_{4} \right)=0.
\end{align}
On the other hand, we may remark that
\begin{align*}
& J(u_{0})= \lambda_{1}^{[1]}u_{1} \Leftrightarrow  D^{2}\left(a_{2}(x) u_{0}\right) = 2 \lambda_{1}^{[1]}u_{1} + 2D\left(a_{0}^{[1]} u_{0} \right),\\
& J(u_{1})= \lambda_{2}^{[1]}u_{2} \Leftrightarrow  D^{2}\left(a_{2}(x) u_{1}\right) = 2 \lambda_{2}^{[1]}u_{2} + 2D\left(a_{0}^{[1]} u_{1} \right),
\end{align*}
enabling the substitution of the second order derivatives of \eqref{n=0andn=1} that together with 
\eqref{Eq-u2}-\eqref{Eq-u4} shall give place to the two next identities.
\begin{align*}
& D\left( \phi_{1,1}(x) u_{0} + \phi_{1,2}(x) u_{1}\right) +u_{1} = 0,\\
& D\left( \phi_{2,1}(x) u_{0} + \phi_{2,2}(x) u_{1}\right) +2E_{1}(x)u_{0}+2A_{0}u_{1} = 0,\;\; \textrm{with}
\end{align*}
\begin{align}
\label{11coeffmatrixEq-F} \phi_{1,1}(x) &= \frac{1}{2} - \frac{\alpha_{1}  \lambda_{2}^{[1]}}{2 \lambda_{1}^{[1]}} E_{1}(x) - \frac{\gamma_{1}  \lambda_{3}^{[1]}}{2 \lambda_{1}^{[1]}} B_{1}(x) \\
\label{12coeffmatrixEq-F} \phi_{1,2}(x) &= \frac{1}{2}\left(x-\beta_{0} \right) - \frac{\alpha_{1}  \lambda_{2}^{[1]}}{2 \lambda_{1}^{[1]}} A_{0}(x) - \frac{\gamma_{1}  \lambda_{3}^{[1]}}{2 \lambda_{1}^{[1]}} F_{1}(x) \\
\label{21coeffmatrixEq-F} \phi_{2,1}(x) &= \left(x-\beta_{1} \right)E_{1}(x)-\frac{\alpha_{2}  \lambda_{3}^{[1]}}{ \lambda_{2}^{[1]}} B_{1}(x)-\frac{\gamma_{2}  \lambda_{4}^{[1]}}{ \lambda_{2}^{[1]}} E_{2}(x)\\
\label{22coeffmatrixEq-F} \phi_{2,2}(x) &= \left(x-\beta_{1} \right)A_{0}(x)-\frac{\alpha_{2}  \lambda_{3}^{[1]}}{ \lambda_{2}^{[1]}} F_{1}(x)-\frac{\gamma_{2}  \lambda_{4}^{[1]}}{ \lambda_{2}^{[1]}} A_{1}(x).
\end{align}
In particular, we have assured that
$\deg \left( \phi_{1,1}(x)\right) \leq 1$, $\deg \left( \phi_{1,2}(x)\right) \leq 1$, $\deg \left( \phi_{2,1}(x)\right) \leq 2$ and $\deg \left( \phi_{1,2}(x)\right) \leq 1$  as required in \cite{d-classical}.
\end{proof}

\section{Conclusion and future developments}

We have learned in sections two and three, the operator $J$ is stretchy enough to become some  of the most famous operators that do not increase the degree of polynomials.  In addition, the problem of finding the polynomial sequences that coincide with the (normalized) $J$-image of themselves, or more briefly, the sequences $\textbf{P} = \left( P_{0}, P_{1}, \ldots \right)^{t}$ such that $\textbf{P} = \Lambda^{[k]}_{J} \; J \left( \textbf{P} \right),$ for a certain matrix $\Lambda^{[k]}_{J}$ (with infinite dimensions) filled with the necessary normalization coefficients, may be transposed to the analysis of functional identities.
\newline This theoretical exposition is then firstly put in practice with a simple case in the fourth section. More precisely, we have chosen an operator $J$ with only three terms  $J = a_{0}(x)I+a_{1}(x)D + \frac{a_{2}(x)}{2}D^{2}$. With this choice, we may ponder the question above for three orders of differentiation of $J$: $k=0$, $k=1$ and $k=2$. When $k=0$, our debate goes around the polynomial coefficients of the differential equation fulfilled by the four family of classic MOPSs, whereas when $k=1$, the solution belongs to the classic Laguerre family enlarging the prior results of \cite{QuadraticAppell, MesquitaPMH},  and finally when $k=2$, the single solution is indeed an Hermite form.
 A brief introduction to the analysis of a two-orthogonal MPS is provided in section \ref{sec:twoorto} for a lowering operator of order one.
\newline Although most conclusions regarding the orthogonal hypothesis go along with the literature, the different path here taken puts in evidence a common approach for different types of differential operators and allows the rephrasing of a well known environment (as the one of orthogonal sequences and $J$ up to three terms). The initial approach to a two-orthogonal situation let us aspire that in the near
future we may apply it when other operators $J$ are chosen and when the polynomial sequence $\{P_{n}\}_{n \geq 0}$ is, for instance, $d$-orthogonal with $d>1$ also known as the step-line multiple orthogonal polynomials.

\section*{Acknowledgments}

\noindent 
T. Augusta Mesquita was partially supported by CMUP (UID/MAT/00144/2013), which is funded by FCT (Portugal) with national (MEC) and European structural funds (FEDER), under the partnership agreement PT2020.


\end{document}